\documentclass[11pt]{amsart}
\usepackage{amsmath,amssymb}
\usepackage{color}
\usepackage[all,curve,frame]{xy}

\usepackage[pdftex]{graphicx}

\newtheorem{defi}{Definition}[section]
\newtheorem{sinobservacion}[defi]{}
\newenvironment{sinob}{\begin{sinobservacion} \rm}{\end{sinobservacion} }
\newtheorem{coro}[defi]{Corollary}
\newtheorem{lema}[defi]{Lemma}
\newtheorem{nota}[defi]{Notation}
\newtheorem{obs}[defi]{Remark}
\newtheorem{prop}[defi]{Proposition}
\newtheorem{teo}[defi]{Theorem}
\newtheorem{ej}[defi]{Example}

\newcommand{\benu}{\begin{enumerate}}
\newcommand{\enu}{\end{enumerate}}

\newcommand{\al}{\alpha}

\newcommand{\mbmA}{\mbox{mod}\,A}
\newcommand{\emmA}{\emph{mod}\,A}
\newcommand{\mbiA}{\mbox{ind}\,A}

\newcommand{\mbmB}{\mbox{mod}\,B}
\newcommand{\emmB}{\emph{mod}\,B}

\newcommand{\mbeT}{\mbox{End}_AT}
\newcommand{\emeT}{\emph{End}_AT}
\newcommand{\Gaa}{\Gamma_A}

\newcommand{\fle}{\rightarrow}
\newcommand{\flechad}{\rightarrow}
\newcommand{\caminod}{\rightsquigarrow}
\newcommand{\F}{\mathcal{F}}
\newcommand{\T}{\mathcal{T}}
\newcommand{\X}{\mathcal{X}}
\newcommand{\Y}{\mathcal{Y}}

\newcommand{\til}{\widetilde}
\newcommand{\phii}{\varphi}

\begin{document}
\title[On the radical of a module category]
{On the radical of the module category of an endomorphism algebra}
\author[Chaio Claudia]{Claudia Chaio}
\address{Centro marplatense de Investigaciones Matem\'aticas, Facultad de Ciencias Exactas y
Naturales, Funes 3350, Universidad Nacional de Mar del Plata, 7600 Mar del
Plata, Argentina}
\email{claudia.chaio@gmail.com}

\author[Guazzelli Victoria]{Victoria Guazzelli}
\address{Centro marplatense de Investigaciones Matem\'aticas, Facultad de Ciencias Exactas y
Naturales, Funes 3350, Universidad Nacional de Mar del Plata, 7600 Mar del
Plata, Argentina}
\email{victoria.guazzelli@hotmail.com}
\maketitle

\begin{abstract}
Given a finite dimensional algebra $A$ over
an algebraically closed field we study the relationship between the powers
of the radical of a morphism in the module category of the algebra $A$ and the induced morphism in
the module category of the endomorphism algebra of a tilting $A$-module.
We compare the nilpotency indices  of the radical of the mentioned module categories.
We find an upper bound for the nilpotency index of the radical of the module category of
iterated tilted algebras of Dynkin type.
\end{abstract}

\section*{Introduction}
Let $A$ be an a finite dimensional $k$-algebra over an algebraically closed field $k$, and $\mbox{mod}\;A$ the
category of finitely generated left $A$-modules. For $X,Y \in \mbox{mod}\;A$,
we denote by $\Re(X,Y)$ the set of
all morphisms $f: X \rightarrow Y$ such that, for all
indecomposable $A$-module $M$, each pair of morphisms $h:M
\rightarrow X$ and $h':Y\rightarrow M$ the composition $h'fh$ is not an isomorphism.
Inductively, the powers of
$\Re(X,Y)$ are defined. By $\Re_A^\infty(X,Y)$
we denote the intersection of all powers $\Re_A^i(X,Y)$ of $\Re_A(X,Y)$, with $i \geq 1$.

An important research direction towards understanding the structure of a module category is the study of the compositions of irreducible morphisms
in relation with the powers of the radical of their module categories, see for example \cite{C3}, \cite{CLT} and \cite{ChaTre}.

In case we deal with a representation-finite algebra, it is well-known
that there is a positive integer $n$ such that $\Re^{n}(\mbox{mod}\,A)=0$, see \cite[p. 183]{ARS}.
The minimal lower bound $m \geq 1$ such that $\Re^{m}(\mbox{mod}\;A)$ vanishes
is called the nilpotency index of $\Re(\mbmA)$.

In \cite{C}, the first named author determined the
nilpotency index of $\Re(\mbmA)$  in terms of the left and right degrees of same particular irreducible morphisms.
The notion of degree of an irreducible morphisms was introduced by S. Liu in \cite{L}.
In this work, we improve such  a result. More precisely, we reduce the steps to compute such a bound in case we deal with algebras such that all vertices in the ordinary quiver are neither sinks nor sources, see Theorem \ref{nilpo}.

One of the aims of  this article is to study the relationship between the power of the radical
of the morphisms in $\mbmA$ and the  induced morphisms in $\mbmB$, where $B$
is the endomorphism algebra of some tilting $A$-modules.
To achieve to this result, we strongly use the well-known tilting theorem due to Brenner and Butler.

We also compare the nilpotency indices of the radical of the module
category of an algebra and of the endomorphism algebra of some specific tilting $A$-modules.
Precisely, we prove the following result.
\vspace{.05in}

{\bf Theorem A}. {\it
Let $A\simeq kQ_A/I_A$ be a representation-finite  algebra, $T$ be a separating tilting $A$-module and
$B=\emeT$. Consider $(R_A)_0$ to be the set of all the vertices $u\in Q_A$
such that the length of any path from $P_u$ to $I_u$ going through $S_u$ is $r_A-1$, where $r_A$ is the nilpotency index of $\Re(\emmA)$ and $P_u$, $I_u$ and $S_u$ are the projective, the injective and the simple  corresponding to the vertex $u\in Q_A$,  respectively. If $B$ is representation-finite and for some $u\in ({R_A})_0$
we have that $P_u\in \emph{add}\,T$ then $r_A\leq r_B,$
where $ r_B$ is the nilpotency index of $\Re(\emmB)$.}
\vspace{.05in}

As a consequence we obtain information between the nilpotency index of the radical of the module
category of an algebra and of the endomorphism algebra of an APR-tilting module,
when both algebras are of finite representation type.

In \cite{Z},  D. Zacharia determined the nilpotency index of the radical of the module
category of any representation-finite hereditary algebra in terms of the number of vertices of their ordinary quiver.

For an iterated tilted algebra of Dynkin type, we  find an upper bound of
such an index. In many cases, this bound coincide with the nilpotency index, see Example \ref{ejemplo1}. More precisely, we  prove Theorem B.
\vspace{.05in}

{\bf Theorem B}. {\it Let $\Delta$ be a Dynkin quiver and $A$ be an iterated tilted algebra of type $\Delta.$
Consider $r_A$ the nilpotency index of $\Re_A(\emmA)$ and $\overline{\Delta}$ the underlying graph.
\begin{enumerate}
\item[(a)] If $\overline{\Delta}=A_n$, then $r_A\leq n$ for $n\geq 1$.
\item[(b)] If $\overline{\Delta}=D_n$, then  $r_A\leq 2n-3$ for $n\geq 4$.
\item[(c)] If $\overline{\Delta}=E_6$, then  $r_A\leq 11.$
\item[(d)] If $\overline{\Delta}=E_7$, then  $r_A\leq 17.$
\item[(e)] If $\overline{\Delta}=E_8$, then  $r_A\leq 29.$
\end{enumerate}}

\vspace{.05in}

This paper is organized as follows. In the first section, we state some notations and recall some preliminaries results.
Section 2 is dedicated to prove the results concerning
the comparison between morphisms in $\mbmA$ and the induced morphisms in $\mbmB$,
in relation with the power of the radical of their module categories. We prove Theorem A.
In Section 3, we
apply the results concerning the nilpotency index of the radical of the module
category of a representation-finite hereditary algebra to find an upper bound for nilpotency index of the radical of the module
category of a iterated tilted algebras of Dynkin type. Precisely, we prove Theorem B.
\vspace{.1in}

\thanks{Both authors thankfully acknowledge partial support from CONICET
and from Universidad Nacional de Mar del Plata, Argentina. The first author is a researcher from CONICET. }

\section{preliminaries}

Throughout this work, by an algebra we mean a finite dimensional basic $k$-algebra over an algebraically closed field,  $k$.

\begin{sinob} A {\it quiver} $Q$ is given by a set of vertices $Q_0$ and
a set of arrows $Q_1$, together with two
maps $s,e:Q_1\fle Q_0$. Given an arrow $\al\in Q_1$, we write $s(\al)$
the starting vertex of $\al$ and $e(\al)$
the ending vertex of $\al$.
If  $A$ is an algebra then
there exists a quiver $Q_A$, called the {\it ordinary quiver of}  $A$, such that
$A$ is the quotient of the path algebra $kQ_A$ by an admissible ideal.
\end{sinob}

\begin{sinob}
Let $A$ be an algebra. We denote by $\mbox{mod}\,A$ the category of finitely generated
left $A$-modules and  by $\mbiA$ the full subcategory of $\mbmA$ which consists of
one representative of each isomorphism class of indecomposable $A$-modules.

For any $A$-module $M$, we denote by $\mbox{add}\,M$  the full subcategory of $\mbmA$ whose objects are the direct sums of summands of $M$.

We denote by $\Gaa$ the Auslander-Reiten
quiver of $\mbmA$, by $\tau$ and $\tau^{-1}$ the Auslander-Reiten translation $\mbox{DTr}$ and $\mbox{TrD}$, respectively.
\end{sinob}

\begin{sinob}
Consider  $X,Y \in \mbox{mod}\,A$. The ideal $\Re(X,Y)$ is
the set of all the morphisms $f: X \rightarrow Y$ such that, for each
$M \in \mbox{ind}\,A$, each $h:M \rightarrow X$ and each $h^{\prime }:Y\rightarrow M$
the composition $h^{\prime }fh$ is not an isomorphism. For $n \geq 2$, the powers of $\Re(X,Y)$
are defined inductively. By $\Re^\infty(X,Y)$ we denote the intersection
of all powers $\Re^i(X,Y)$ of $\Re(X,Y)$, with $i \geq 1$.

Following \cite{CL}, we say that the depth of a morphism $f: M \rightarrow N$
in ${\rm mod}\,A$ is infinite if $f\in \Re^\infty(M, N)$; otherwise, the depth
of $f$ is the integer $n\geq 0$ for which $f\in \Re^n(M, N)$ but $f \notin \Re^{n+1}(M, N)$.
We denote the depth of $f$ by ${\rm dp}(f)$.

A morphism $f :X  \rightarrow  Y$, with $X,Y \in \mbox{ mod}\,A$,
is called {\it irreducible} provided it does not split
and whenever $f = gh$, then either $h$ is a split monomorphism or $g$ is a
split epimorphism.

By \cite{B}, it is known that a morphism $f :X  \rightarrow  Y$, with $X,Y \in \mbox{ind}\,A$,
is irreducible if and only if ${\rm dp}(f)=1$.

Next, we recall the definition of degree of an irreducible morphism given by S. Liu in \cite{L}, a notion fundamental for this work.

Let $f:X\rightarrow Y$ be an irreducible morphism in
$\mbox{mod}\,A$, with $X$ or $Y$ indecomposable. The {\it left degree} $d_l(f)$ of $f$ is infinite,
if for each integer $n\geq 1 $, each module $Z\in \mbox{ind}\,A$
and each morphism $g:Z  \rightarrow X$ with ${\rm dp}(g)=n$ we
have that $fg \notin \Re^{n+2}(Z,Y)$. Otherwise, the left degree of
$f$ is the least natural $m$ such that there is an $A$-module $Z$
and a morphism $g:Z  \rightarrow X$ with ${\rm dp}(g)=m$ such that
$ fg \in \Re^{m+2}(Z,Y)$.

The {\it right degree} $d_r(f)$ of an irreducible morphism $f$
is dually defined.

We say that an algebra A is {\it representation-finite} if there is only a finite number
of isomorphisms classes of indecomposable A-modules.  It is well-known that an algebra $A$ is representation-finite
if and only if there is a positive integer $n$ such that $\Re^n(X,Y)=0$ for each $A$-module $X$ and $Y$.
The minimal lower bound $m$ such that $\Re^m(\mbmA)=0$ is called the \textit{nilpotency index} of
$\Re(\mbmA)$. We denote it by $r_A$.

If $A$ is representation-finite, by \cite{CLT} we know  that all
irreducible epimorphisms and all irreducible monomorphisms are of
finite left and right degree, respectively.
We denote by $P_a$, $I_a$ and  $S_a$ the projective, the injective and the simple $A$-module corresponding to the
vertex $a$ in $Q_A$, respectively.

In \cite{C}, the first named author, find the nilpotency index of a representation-finite algebra in terms of degrees of some irreducible morphisms and also in terms of length of paths. Precisely, the author proved the following result, fundamental for our purposes.

If either
$P_a=S_a$ or $I_a=S_a$ then we write $n_a=0$ or $m_a=0$,
respectively. Otherwise, we consider the irreducible morphisms
$\iota_a: \mbox{rad} (P_a)\hookrightarrow P_a$ and $g_a:I_a
\longrightarrow I_a/\mbox{soc}(I_a)$ and we write
$n_a=d_r(\iota_a)$ and
$m_a=d_l(g_a)$. We write $r_a=n_a+m_a$.

\noindent\begin{teo}\emph{({\cite[ Theorem A]{C}.})}  Let $A\simeq kQ_A/I_A$ be a finite
dimensional algebra over an algebraically closed field and assume
that $A$ is of finite representation type. Then, the nilpotency index $r_A$ of the radical of the module category of $\emph{mod}\,A$ is
$\;\mbox{max}\; \{r_a\}_{a \in Q_0}+1$.
\end{teo}

A {\it path}  in $\mbmA$ is a sequence of non-zero non-isomorphisms between indecomposable $A$-modules.
In case that the morphisms involved in the above path are irreducible, we say that it is a path in $\mbox{ind}\; A$.

The {\it length} of a path $M_0\stackrel{f_{1}}\fle M_1\stackrel{f_{2}}\fle M_2\fle  \dots \fle M_{t-1}\stackrel{f_{t}}\fle M_t$,
between indecomposable $A$-modules is the number $t$ of irreducible morphisms in it.
Given $M, N\in \mbiA$, we denote a path in  $\mbox{ind}\; A$ from $M$ to $N$  by  the expression  $M\caminod N$.

By \cite[Remark 1]{C}, we know that $r_a$ is equal to the length of any path of irreducible
morphisms between indecomposable modules from the projective $P_a$ to the injective
$I_a$, going through the simple $S_a$. Therefore, in order to find the nilpotency index we have to find the path of maximum length between them.

For the convenience of the reader we recall a result from \cite{C} that we shall use frequently in this work.

\begin{lema}\label{clau2.5}\emph{\cite[Lema 2.4]{C}}
Let $A\cong kQ/I$
 be a representation-finite algebra.
 Given $a\in Q_{0}$, consider $r_a$ the number defined above.
Then,\begin{enumerate}
\item[(a)]
every nonzero morphism $f:P_{a}\rightarrow I_{a}$ that factors through the simple A-module $S_a$ is
such that $dp(f)={r_a}$.
\item[(b)] every non-zero morphism $f:P_{a}\rightarrow I_{a}$ which does not factor through the simple A-
module $S_a$ is such that $dp(f)={k}$, with $0\leq k<r_a$.
\end{enumerate}
\end{lema}

\end{sinob}

\begin{sinob} Let $\Gamma$ be a component of $\Gamma_A$. Following \cite{CPT}, we say that $\Gamma$ is a component {\it with length}
if parallel paths in $\Gamma$ have the same length. By parallel paths we mean paths in $\mbox{ind}\; A$  having the
same starting vertex and the same ending vertex.

Given a directed component $\Gamma$ of $\Gamma_A$, its {\it orbit graph}, denoted by $O(\Gamma )$, has as points the $\tau$ -orbits
$O(M)$ of the modules M in $\Gamma$ . There exists an edge between $O(M)$ and $O(N)$ in $O(\Gamma )$
if there are integer numbers $m$ and $n$  and an irreducible morphism $\tau^mM \fle \tau^n N$ or $\tau^nN \fle \tau^m M$.
Recall that if the orbit graph  $O(\Gamma)$ is of tree-type,
then $\Gamma$ is a simply connected translation quiver, and therefore, by \cite{BG},
$\Gamma$ is a component with length.

For general references on radical theory we refer the reader to \cite{ARS}.
\end{sinob}

\begin{sinob}
A pair $(\T,\F)$ of full subcategories of
$\mbox{mod}\;A$ is called a \textit{torsion theory} if the following
conditions are satisfied:\begin{enumerate}
\item[(a)] $\mbox{Hom}_{A}(M,N)=0$ for all $M\in \T$ y $N\in \F$.
\item[(b)] If $\mbox{Hom}_{A}(M,F)=0$ for all $F\in \F$, implies $M\in \T$.
\item[(c)] If $\mbox{Hom}_{A}(T,N)=0$ for all $T\in \T$, implies $N\in \F$.
\end{enumerate}

According to \cite{HR}, an A-module $T$ is called a tilting module if $\mbox{pd}\,T\leq 1$,
$\mbox{Ext}_{A}^{1}(T,T)=0$
and the number of isomorphism classes of
indecomposable summands of $T$ equals the number of vertices in $Q_A.$

Given $T$ a tilting
$A$-module, there exists a close connection between the representation theories
of $A$ and of $B=\mbox{End}_A(T)$.
It is well-known that a tilting $A$-module induces
a torsion pair $({\mathcal{T}}(T), {\mathcal{F}}(T))$ in $\mbox{mod}\,A$
and  a torsion pair $({\mathcal{X}}(T),{\mathcal{Y}}(T))$ in $\mbox{mod}\,B$.
More precisely, it was shown in \cite{BB} and \cite{HR} that if $T$ is a tilting module and
$B=\mbox{End}_{A}(T)$, then $T$ is a tilting $B$-module and $A\simeq \mbox{End}_{B}T$
and the functors $\mbox{Hom}_{A}(T,-)$ and $-\otimes _{B}T$ induce inverse equivalences between
the full subcategories ${\mathcal{T}}(T)=\{X:\mbox{Ext}_{A}^{1}(T,X)=0\}$ and
${\Y}(T)=\{Y:\mbox{Tor}^B_1(Y;T)=0\}$, while the functors
$\mbox{Ext}^{1}_{A}(T,-)$ and  $\mbox{Tor}_{1}^{B}(-,T)$ induce inverse equivalences between
the full subcategories ${\mathcal{F}}(T)=\{X: \mbox{Hom}_{A}(T,X)=0\}$ and
${\X}(T)=\{Y : Y\otimes _{B}T=0\}$. This last result is well-known as Brenner and Butler's theorem.

A tilting $A$-module $T$ is called \textit{separating} (\textit{splitting}, respectively)
 if, for each
indecomposable $A$-module $M$ ($B$-module $N$, respectively), we have either $M\in \T(T)$ or
$M\in \F(T)$ ($N\in \X(T)$ or
$N\in \Y(T)$, respectively).

An example, of a separating tilting $A$ module,
is provided by the so-called APR-tilting modules. Let $P_{a}$ be a simple projective
non-injective module (corresponding to the sink $a\in Q_0$).
Then, the module
$T[{a}]= \tau^{-1}(S_{a})\oplus (\oplus_{b\neq a}P_{b}) $
is
called the  \textit{APR-tilting module} associated to the vertex $a$.
Moreover, $\F(T[a])=\mbox{add}S_a$ and
$\T(T[a])=\mbox{add}(\mbiA\backslash S_a).$

In \cite{As}, the author proved that an APR-tilting $A$-module $T[a]$
is splitting if and only if the vertex $a$ is a free sink, that is,
$a$ is not an end point of a zero relation
on $Q_A$. Moreover,
the vertex of $Q_B$ corresponding to $a$ is a source.

For general references on tilting theory we refer the reader to \cite{As},  \cite{ASS} and \cite{HR}.
\end{sinob}

\begin{sinob} \label{uno-ocho}
Consider $H\cong kQ_H$ a representation-finite hereditary algebra,
that is, $Q_H$ is a Dynkin quiver.
We end up this section recalling the definition of an iterated tilting algebra of type $\Delta$.

Consider $\Delta$ a finite connected quiver without oriented cycles. An algebra $A$
is called an \textit{iterated tilting algebra of type $\Delta$} if there exist a sequence
of algebras $A=A_0,A_1,\dots , A_m=k\Delta$ and a sequence of
separating $A_i$-tilting modules $T^{(i)}$, for
$0\leq i<m$, such that $A_{i+1}=\mbox{End}_{A_i}T^{(i)}.$
In particular, if $m=1$ then $A$ is called a \textit{tilted algebra}.
\end{sinob}

\section{On the nilpotency index}

Throughout this work, by $r_A$ we denote the nilpotency index of the radical of the module category of a representation-finite algebra and by $r_B$ the nilpotency index of the radical of the module category of the  endomorphism algebra of a tilting $A$-module.

The aim of this section is to compare the nilpotency indices of the radical of the module category of a given algebra $A$ and the module category of the endomorphism algebra of a tilting $A$-module.

First, we establish the relationship between the powers
of the radical of a morphism in the module category of $A$ and the induced morphism in
the module category of the endomorphism algebra of a tilting $A$-module.

In the next result, we study the situation when $T$ is a separating tilting module. A dual result holds for
a splitting tilting module.

\begin{prop}\label{Tsep}
Let $A$ be an algebra, $T$ be a separating tilting $A$-module and
$B=\emph{End}_{A}(T)$. Let $M,N$ be indecomposable $A$-modules in $\T(T)$.
If  $f\in \Re_A^n(M,N)$ for some $n\geq 1$
then $F(f)\in \Re_B^n(F(M),F(N))$,
where  $F=\emph{Hom}_A(T,-)$.

In addition, if  $dp(F(f))=n$
then there is a non-zero path of irreducible morphisms between indecomposable $B$-modules in $\Y(T)$ from  $F(M)$ to $F(N)$ of length $n$.
\end{prop}

\begin{proof}
Let $M,N$  be indecomposable $A$-modules in $\T(T)$  and $f\in \Re_A^n(M,N)$, for some $n\geq 1$.
We prove by induction on $n$  that $F(f)\in \Re_B^n(F(M),F(N))$.

If $n=1$ then $f$  is not an isomorphism. Hence, by Brenner and Butler's theorem  it follows that $F(f)$ is not an
isomorphism in $\mbmB$. Therefore,  $F(f)\in \Re_B(F(M),F(N))$.

Now, consider $n\geq 2$. By \cite[Proposition 7.4]{ARS}, there exists a positive integer $s\geq 1$
and indecomposable $A$-modules $X_1, \dots, X_s$ such that
 $f=\Sigma_{i=1}^s g_if_i$, with  $f_i\in \Re_A(M,X_i)$
and $g_i=\Sigma_{k=1}^{t_i} h_i^k$, where each $h_i^k$
is a composition of $n-1$ irreducible morphisms between indecomposable $A$-modules,
for $1\leq i\leq s$, $1\leq k\leq t_i$ and $1\leq j\leq n$.
Therefore, $g_i\in \Re_{A}^{n-1}(X_i,N)$ for $1\leq i\leq s$. Without loss of generality, since $f\neq 0$,
we may assume that for all $i$, the morphisms $f_i$ and $g_i$ do not vanish.

Since $T$ is a separating tilting $A$-module, then all
the indecomposable modules $X_i$, with $1\leq i\leq s$,
belong to $\T(T)$ because $M \in \T(T)$ and $\mbox{Hom} (\T(T), \F(T))=0$. %there is not
%non-zero morphisms from $\T(T)$ to $\F(T)$.
Hence, by Brenner and Butler's theorem  we have that
$F(f)=\Sigma_{i=1}^s F({g}_i)F({f}_i)$, with  $F({f}_i)\in \Re_A(F({M}),F({X}_i))$. Then, by inductive hypothesis $F({g}_i)\in \Re_B^{n-1}(F({X_i}),F(N))$.
Therefore $F({f})\in \Re_B^n(F({M}),F({N}))$.

Now, assume that $F(f)\notin \Re_B^{n+1}(F(M),F(N))$. Thus $f\notin\Re_A^{n+1}(M,N)$.
If $n=1$ then $F(f)$ is irreducible and there is a path  from
$F(M)$ to $F(N)$ of length one between modules in $\Y(T)$.

If $n\geq 2$, since $dp(f)=n$,
then by \cite[Proposition 7.4]{ARS}, there exists a positive integer $s\geq 1$ such that
$f=\Sigma_{i=1}^s \Sigma_{k=1}^{t_i} h_i^kf_i$, where each $h_i^k=h_i^{k,n-1} \dots h_i^{k,1}$
is a composition of $n-1$ irreducible morphisms between indecomposable modules in $\T(T)$.
Then, $$F(f)=\Sigma_{i=1}^s \Sigma_{k=1}^{t_i} F(h_i^{k,n-1})\dots F(h_i^{k,1})F(f_i).$$
\noindent If all the summands of $F(f)$ belong to $\Re_B^{n+1}(F(M),F(N))$,
then $F(f) \in \Re_B^{n+1}(F(M),F(N))$,
getting a contradiction to our  assumption.
Hence, there exist at least an $i_0 \in \{1, \dots, s\}$ and a $k_0 \in \{1,  \dots, t_{i_0}\}$
such that $dp(F(h_{i_0}^{{k_0},n-1}) \dots F(h_{i_0}^{{k_0},1})F(f_{i_0}))=n$.
Thus, it is a composition of $n$ irreducible morphisms between indecomposable $B$-modules
in  $\Y(T)$, proving the result.
\end{proof}

The above result holds if we consider a morphism between indecomposable
modules in the free-torsion class. Bellow, we state the result.

\begin{prop}\label{Tsepdual}
Let $A$ be an algebra, $T$ be a separating tilting $A$-module and
$B=\emph{End}_{A}(T)$. Let $M,N$ be indecomposable $A$-modules in $\F(T)$.
If  $g\in \Re_A^n(M,N)$ for some $n\geq 1$
then $F'(g)\in \Re_B^n(F'(M),F'(N))$, where  $F'=\emph{Ext}^1_A(T,-)$.

In addition, if  $dp(F'(g))=n$
then there is a non-zero path of irreducible morphisms between indecomposable $B$-modules in $\X(T)$ from $F'(M)$ to $F'(N)$ of length $n$.
\end{prop}

Our next examples show that the above result can not be improved taking into account the depth of the morphism.

\begin{ej} \label{eje1}
\emph{$(a)$ First, we show an example of  a morphism which does not belong to $\Re^2_A$  but where the induced morphism is in $\Re^2_B$.}

\emph{ Let $A$  be an algebra given by the quiver}

\begin{displaymath}
	\xymatrix  @R=0.3cm  @C=0.3cm {
		 1\ar[rr]^{\alpha}& & 2\ar[rr]^{\beta}&& 3\ar[rr]^{\gamma}&&4 &}
	\end{displaymath}

\noindent \emph{ with $I_A=<\gamma\beta\alpha>$. The Auslander-Reiten quiver of  $\mbmA$ is the following:}

\[
\includegraphics[width=11cm]{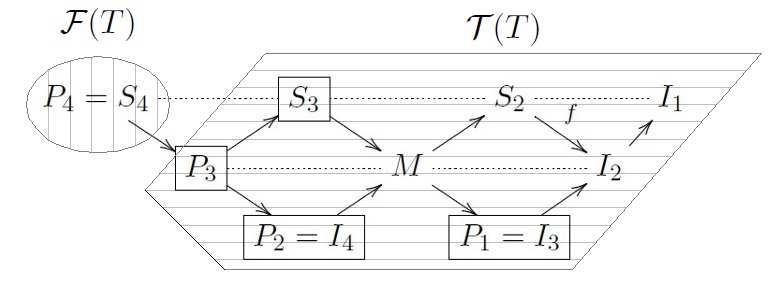}
\]

\noindent  \emph{where
$T=P_1\oplus P_2\oplus P_3\oplus S_3$
is a tilting $A$-module. The  endomorphism algebra  $B=\mbox{End}_AT$ is given by the bound quiver:}

{\begin{displaymath}
	\xymatrix  @R=0.4cm  @C=0.7cm {
		                   &&& &2'\ar[rd]^{\beta'}&&&\\
 &&& 1' \ar[rd]_{\alpha'}\ar[ru]^{\gamma'}\ar@{.}[rr]&&3'&&&&\\
 &&& &4'\ar[ru]_{\delta'}&&&}
	\end{displaymath}}

\noindent \emph{with  $I_B=<\beta'\gamma'-\delta'\alpha'>$. The Auslander-Reiten quiver of $\mbmB$ is the following}

\[
\includegraphics[width=12cm]{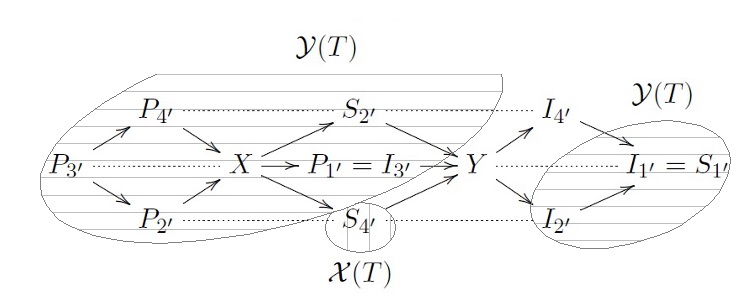}
\]

\noindent \emph{Observe that $T$ is a separating but not splitting tilting $A$-module.
The morphism $f:S_2\fle I_2$ does not belong to $\Re^2_A(S_2,I_2),$
while the induced morphism $\mbox{Hom}_A(T,f)$ is in $\Re^2_B(S_{2'},I_{2'})$.}

\emph{ $(b)$ Our next example shows a morphism in $\mbmA$ such that $dp(f)=t$ and where the induced morphism
belongs to $\Re^{\infty}_B$.}

\emph{Let $A$  be the algebra given by the quiver}

\begin{displaymath}
	\xymatrix  @R=0.1cm  @C=0.3cm {
&		&&2\ar[rrdd]^{\beta}&&&&\\
&&\ar@{.}@/_/[rr]&&&&&\\
&1\ar[rrdd]_{\delta}\ar[rruu]^{\alpha}&&&&4&&\\
&&\ar@{.}@/^/[rr]&&& &&\\
&&&3\ar[rruu]_{\gamma}&&&&}
	\end{displaymath}

\emph{\noindent with $I_A=<\beta\alpha,\gamma\delta>$. The Auslander-Reiten quiver of  $\mbmA$ is the following:}

\[
\includegraphics[width=12cm]{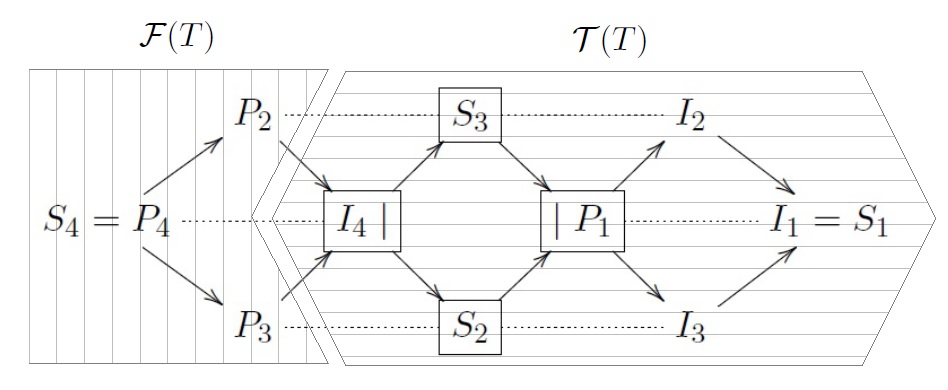}
\]

\noindent \emph{Consider the tilting module $T=I_4\oplus S_3 \oplus S_2 \oplus P_1$.
The endomorphisms algebra $B=\mbox{End}\,T$ is given by the quiver}

\begin{displaymath}
	\xymatrix  @R=0.1cm  @C=0.3cm {
&		&&2'\ar[rrdd]^{\beta}&&&&\\
&&&&&&&\\
&1'\ar[rrdd]_{\delta}\ar[rruu]^{\alpha}&&&&4'&&\\
&&&&& &&\\
&&&3'\ar[rruu]_{\gamma}&&&&}
	\end{displaymath}

\emph{Note that $B$ is a hereditary algebra of type $\til{A}_3$. Consider the morphism
$f:P_1\fle I_1$ in $\mbmA$. By \ref{clau2.5} we know that
$\mbox{dp}(f)=2$. On the other hand, if we applied to $f$ the functor $F=\mbox{Hom}_A(T,-)$, the morphism $F(f)$
belongs to  $\Re^{\infty}_B(F(P_1),F(I_1))$, since
$F(P_1)$ and $F(I_1)$ is a projective and an injective $B$-module, respectively,
in different Auslander-Reiten components of $\mbmB$.}
\end{ej}

As a consequence of Proposition \ref{Tsep} and Proposition \ref{Tsepdual}, we deduce the following corollary.

\begin{coro}\label{BBcompos}
Let $A$ be an algebra, $T$ be a separating and splitting tilting $A$-module and
$B=\emph{End}_{A}(T)$. Let $M,N$ be indecomposable $A$-modules in $\T(T)$ (or in $\F(T)$).
Then, for some $n\geq 1$,  $f\in \Re_A^n(M,N)$  if and only if
$F(f)\in \Re_B^n(F(M),F(N))$ (or $F'(f) \in \Re_B^n(F'(M), F'(N))$),
where  $F=\emph{Hom}_A(T,-)$ (\text{or}  $F'=\emph{Ext}_A^{1}(T,-))$.
\end{coro}

Next, we show an example where the above results do not hold if we consider a morphism
$f:M\fle N,$ where the domain and codomain belong to different subcategories.

\begin{ej}
\emph{Let $A$  be the algebra given by the quiver}

\begin{displaymath}
\xymatrix  @R=0.5cm  @C=0.6cm {
1\ar[r]& 2\ar[r]& 3\ar[r]& 4& }
\end{displaymath}

\emph{The Auslander-Reiten quiver of  $\mbmA$ is the following:}

\[
\includegraphics[width=11cm]{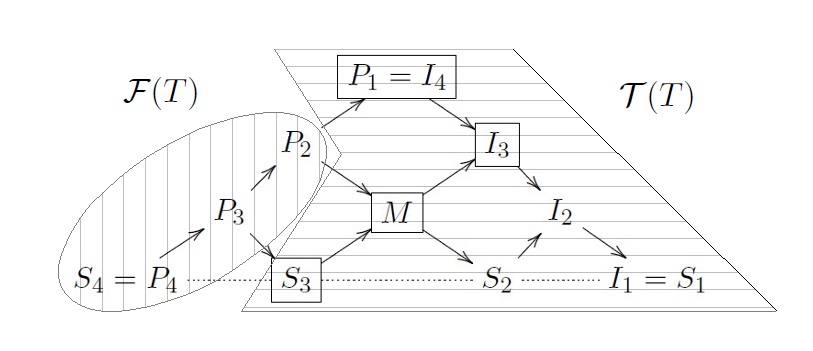}
\]

\noindent\emph{where $T=P_1\oplus I_3\oplus M\oplus S_3$ is a tilting $A$-module.}
\emph{The endomorphism algebra $B=\mbox{End}_AT$ is given by the quiver:}

\begin{displaymath}
\xymatrix  @R=0.6cm  @C=0.7cm {
1'& 2'\ar[r]\ar[l]& 3'\ar[r]& 4'& & }
\end{displaymath}

\noindent \emph{and the Auslander-Reiten quiver of $\mbmB$ is the following}

\[
\includegraphics[width=12cm]{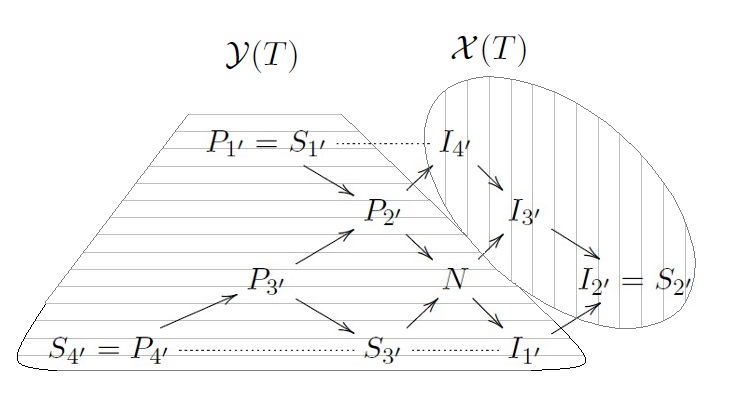}
\]

\emph{Now, if we consider an irreducible morphism $f:P_2\fle P_1$ in $\mbmA$, where $P_2\in \F(T)$ and $P_1\in \T(T)$,
we have that $dp(f)=1$.}

\emph{ On the other hand, $\mbox{Ext}_A^1(T,P_2)\simeq S_{2'}$ and $\mbox{Hom}_A(T,P_1)\simeq S_{1'}$
are non-isomorphic simple $B$-modules, and the unique morphism between them
is the zero morphism, which belongs to the infinite radical.}
\end{ej}

As an immediate consequence of Corollary \ref{BBcompos}, we get the following result.

\begin{coro}\label{BBirred}
Let $A$ be an algebra, $T$ be a separating and splitting tilting $A$-module and
$B=\emph{End}_{A}(T)$.
Then, $f:M\rightarrow N$ is an irreducible morphism between indecomposable $A$-modules
in ${\T}(T)$ (${\F}(T)$, respect.) if and only if $F(f):F(M)\rightarrow F(N)$
($F'(f):F'(M)\rightarrow F'(N)$, respect.)
is an irreducible morphism between indecomposable $B$-modules in
${\Y}(T)$ (${\X}(T)$, respect.), where $F=\emph{Hom}_A(T,-)$ ($F'=\emph{Ext}^1_A(T,-)$, respect.).
\end{coro}

Our next aim is to compare the nilpotency indices
of the radical of the module category of an algebra and  the module category of the endomorphism
algebra of some particular tilting modules.

For any algebra $A$ it is well-known that if $T$ is a separating tilting $A$-module and  $B=\mbox{End}_A(T)$ is  representation-finite, then so is $A$.
Conversely, if $T$ is a splitting tilting $A$-module and $B=\mbox{End}_A(T)$ is representation-infinite,
then so is $A$.

\begin{nota}
\emph{Let $A\simeq kQ_A/I_A$ be a representation-finite algebra.
We denote by $(R_A)_0$ the subset of $(Q_A)_0$ of all the vertices that
determine the nilpotency index $r_A$ of $\Re(\mbox{mod}\;A)$, that is,
$(R_A)_0=\{ u\in (Q_A)_0 \mid r_u=r_A-1\}\subset (Q_A)_0$ where $r_u$ is the length of any path from $P_u$ to $I_u$ going through $S_u$.}
\end{nota}

We are in position to prove Theorem A.

\begin{teo}\label{Tsepar}
Let $A\simeq kQ_A/I_A$ be a representation-finite algebra, $T$ be a separating tilting $A$-module and
$B=\emeT$. If $B$ is representation-finite and for some $u\in ({R_A})_0$
we have that $P_u\in \emph{add}\,T$ then $r_A\leq r_B,$
where $r_A$ and $ r_B$ are the nilpotency indices of $\Re(\emmA)$ and $\Re(\emmB)$,
respectively.
\end{teo}

\begin{proof}
Consider  $A\simeq kQ_A/I_A$ and $B=\mbeT$, with $T$ a separating tilting $A$-module
that satisfies that $P_u\in \mbox{add}\,T$,
for some $u\in ({R_A})_0$. Assume that $B$ is representation-finite then so is $A$. Let $r_A$ and $r_B$ be  the nilpotency index of
$\Re(\mbox{mod}\,A)$ and $\Re(\mbox{mod}\,B)$, respectively.

By hypothesis, the vertex  $u\in (Q_A)_0$ is such that $r_u =r_A-1$.
Consider a non-zero morphism $\phii_u:P_u\fle I_u$ that factors trough $S_u$,
then by Lemma \ref{clau2.5} we have that $dp(\phii_u)={r_A-1}$.
Since $T$ is a separating tilting $A$-module and $P_u, I_u\in\T(T)$,
by Proposition \ref{Tsep} we have that $0\neq F(\phii_u)\in \Re_B^{r_A-1}(F(P_u),F(I_u))$,
where $F=\mbox{Hom}_A(T,-)$. Since  $\Re_B^{r_A-1}(\mbmB)\neq 0$, by the maximality of
$r_B$ we deduce that $r_A\leq r_B$.
\end{proof}

In Theorem \ref{Tsepar} we compare the nilpotency indices of
$\Re(\mbox{mod}\,A)$ and $\Re(\mbox{mod}\,B)$.
Now, we are interested to study when such indices coincide.

Let $S$ be a simple $A$-module,
$P$ and $I$ be the projective cover and the injective envelope of $S$, respectively. Let $T$ be a tilting module
such that $P \in \mbox{add}\,T$ and $B=\mbeT$.
Recall that ${P'}=\mbox{Hom}_A(T,P)$ is an indecomposable
projective $B$-module and that ${I'}=\mbox{Hom}_A(T,I)$ is an indecomposable
injective  $B$-module. Moreover, ${P'}/\mbox{rad}{P'}\cong \mbox{soc}{I'}$.

\begin{prop}\label{propinc}
Let $A\simeq kQ_A/I_A$ be a representation-finite algebra, $T$ be a
separating tilting $A$-module and $B=\emeT$. Assume that
$B$ is representation-finite algebra and that there exists
$u\in ({R_A})_0$ such that $P_u\in \emph{add}\,T$,
where $P_u$ is the projective module corresponding to the vertex $u$.
If $r_A=r_B$, then there exists a non-zero path  of
irreducible morphisms between indecomposable modules from $F(P_u)$ to $F(I_u)$ of length $r_A-1$.
\end{prop}

\begin{proof}
Let $A$ be a representation-finite algebra and $T$ be a separating tilting module, with $P_u\in \mbox{add}\,T$
for some $u\in (R_A)_0$.
Assume that $B=\mbox{End}_A(T)$ is representation-finite and
$r_A=r_B$.
Consider $\phii_u:P_u\fle I_u$ a non-zero morphism
that factors through $S_u$. Thus, $dp(\phii_u)={r_A-1}$
and there is a non-zero path of irreducible morphisms between indecomposable modules,
$P_u\stackrel{f_1} \longrightarrow X_1\stackrel{f_{2}} \longrightarrow \dots \stackrel{f_{r_A-2}} \longrightarrow X_{r_A-2}\stackrel{f_{r_A-1}} \longrightarrow I_u$.
\noindent Since $T$ is a separating tilting module, we claim that for each $i$, the modules $X_i$
belong to the subcategory $\T(T)$, because $P_u\in \T(T)$ and $\mbox{Hom}(\T(T), \F(T))=0$. Following Brenner and Butler's theorem, there exists a non-zero path
of non-isomorphisms as follows:

$$\til{\psi}=F(P_u)\stackrel{F(f_1)} \longrightarrow F(X_1)\longrightarrow  \dots \longrightarrow  F(X_{r_A-2})\stackrel{F(f_{r_A-1})} \longrightarrow F(I_u).$$

\noindent It is clear that every $F(X_i)$ is a $B$-module in $\Y(T)$, for $1\leq i\leq r_A-2$.
Moreover, each $F(f_i)$ is an irreducible morphism. Indeed,
suppose that there exists an integer $j$ with  $1\leq j\leq r_A-1$, such that $F(f_j)$ is not
irreducible. Then, $F(f_j)\in \Re_B^2(F(X_{j-1}),F(X_j))$ and hence
$\til{\psi}\in \Re^{r_A}(\mbmB)=\Re^{r_B}(\mbmB)$, a contradiction to the fact that $\Re^{r_B}(\mbmB)=0$.
Therefore, each morphism $F(f_i)$ is irreducible, proving our result.
\end{proof}

The next examples show that the converse of Proposition \ref{propinc} is not necessarily true.

\begin{ej}
\emph{$(a)$ Consider the algebras $A$ and  $B$ given in Example \ref{eje1} $(a)$. Let $T$ be  the tilting $A$-module, also given in the same mentioned example.
It is not hard to check that the nilpotency index $r_A$ of $\Re(\mbmA)$ is four. Moreover,
the vertex $3$ in $Q_A$  belongs to $(R_A)_0$ and there exists a non-zero path  of
irreducible morphisms between indecomposable modules from $P_{3'}=\mbox{Hom}_A(T,P_3)$ to $I_{3'}=\mbox{Hom}_A(T,I_3)$ of length three.
However, the nilpotency index $r_B$ of $\Re(\mbmB)$ is five. Therefore, $r_A \neq r_B$.}

\emph{ $(b)$ If instead of the existence of a vertex $u \in (R_A)_0$ such that $P_u \in \mbox{add}\,T$ where $T$ a separating tilting $A$-module we put the condition that for all $u \in (R_A)_0$ we have that  $P_u \in \mbox{add}\,T$  then still the converse of  Proposition \ref{propinc} does not hold as we show in our next example.}

\emph{Let $A$ be the algebra given by the quiver}
\begin{displaymath}
	\xymatrix  @R=0.4cm  @C=0.7cm {
&4\ar[dl]_{\alpha}\ar[dr]^{\beta}&&&\\
2\ar[dr]_{\gamma}&&3\ar[dl]_{\delta}&&\\
&1&5\ar[l]^{\mu}&6\ar[l]^{\lambda}}
	\end{displaymath}

\noindent \emph{with $I_A=<\gamma\alpha-\delta\beta, \mu\lambda>$.
The Auslander-Reiten quiver of  $\mbmA$ is the following:}

\[
\includegraphics[width=15cm]{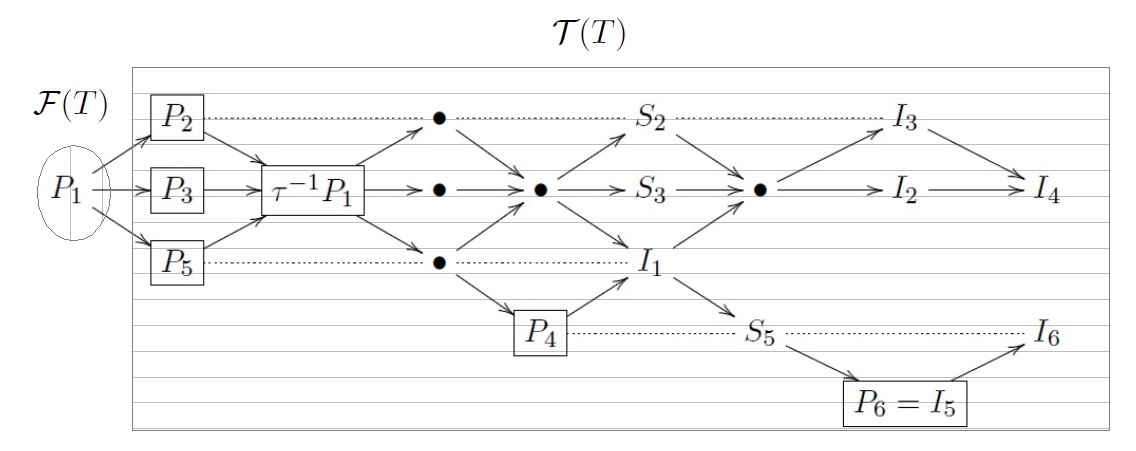}
\]

\noindent\emph{where
$T=\tau P_1\oplus P_2\oplus P_3\oplus P_4\oplus P_5\oplus P_6$
is the APR-tilting $A$-module corresponding to the vertex one.
The  endomorphism algebra  $B=\mbox{End}_AT$ is given by the bound quiver:}
\begin{displaymath}
	\xymatrix  @R=0.5cm  @C=0.4cm {
4'\ar[dr]_{\alpha'}&&6'\ar[dl]^{\beta'}\\
&1'\ar[dl]_{\gamma'}\ar[dr]^{\lambda'}\ar[d]_{\delta'}\\
2'&3'&5'	                 }
	\end{displaymath}
\noindent \emph{with $I_B=<\lambda'\alpha',\gamma'\beta',\delta'\beta'>$,
and the Auslander-Reiten quiver of $\mbmB$ is the following}

\[
\includegraphics[width=15cm]{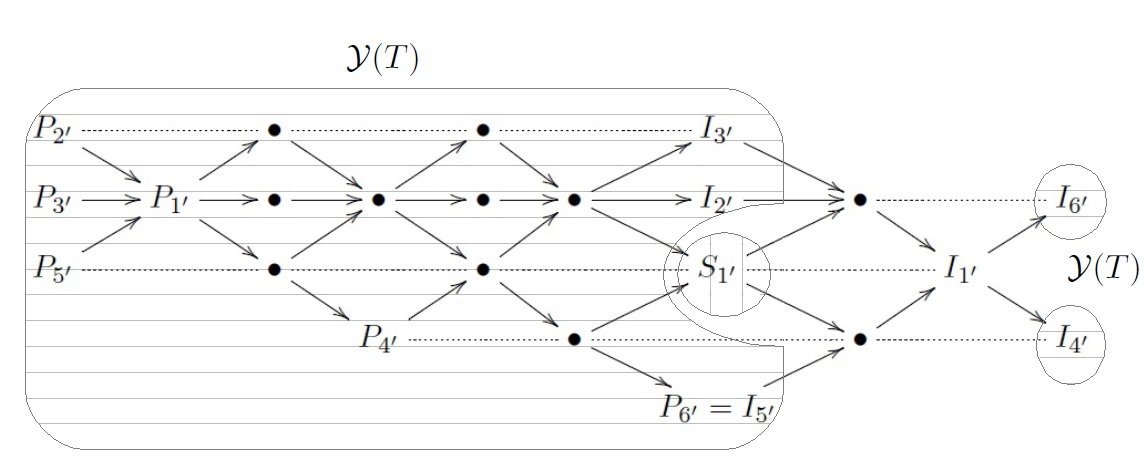}
\]

\emph{Note that the nilpotency index of $\Re(\mbmA)$ is seven. The set $(R_A)_0=\{2,3,5\}$ and
the modules $P_2,P_3$ and $P_5$ belong to the subcategory $\mbox{add}\,T$.}

\emph{On the other hand, the paths in $\mbox{ind}\,B$ from $F(P_i)=P_{i'}$ to $F(I_i)=I_{i'}$ going through the  simple $S_{i'},$ for $i=2,3,5$,
have length six, where $F=\mbox{Hom}_A(T,-)$.  However, the nilpotency index of $\Re(\mbmB)$ is eight, and therefore $r_A\neq r_B$.}
\end{ej}

In the next section we shall prove that if $B$ is an hereditary algebra, hence $A$ is a tilted algebra, then the converse of the implication in the statement of Proposition \ref{propinc} is still true.

Now, we prove a result useful for further considerations.

\begin{prop}\label{nopozo}
Let $A\simeq kQ_A/I_A$ be a representation-finite algebra and $P_a$ be a simple projective $A$-module.  Then, there
is a vertex $b \neq a$ such that   $r_a\leq r_b.$
\end{prop}

\begin{proof}
Since $P_a=S_a$, then any non-zero morphism $\phii_a:P_a\fle I_a$
satisfies that $dp(\varphi_a)={r_a}$.
Hence, by \cite[Proposition 7.4]{ARS}, there is a non-zero path of
irreducible morphisms between indecomposable modules of length $r_a$ as follows:

$$P_a\stackrel{h_1} \longrightarrow X_1\stackrel{h_2}\longrightarrow \dots \stackrel{h_{r_a-1}}\longrightarrow X_{r_a-1}\stackrel{h_{r_a}}\longrightarrow I_a.$$

We claim that $X_1$ is projective. In fact, otherwise,  there is an irreducible morphism $\tau X_1\fle P_a$.
Since $P_a$ is projective, then  $\tau X_1$ is a direct summand of $\mbox{rad} P_a$,
but since $P_a$ is simple we get to a contradiction.
Hence we write $X_1=P_b$, for some vertex $b\neq a$.

Consider $\psi= h_{r_a} \dots h_2: P_b \fle I_a$. Then, $dp(\psi)={r_a-1}.$
Since $b\neq a$, we have that $\psi$ does not factor through $S_b$.
Following \cite[Lemma 2.2]{C}, we know that there is a non-zero morphism $\xi\in \Re_A(I_a,I_b)$
such that the composition $\xi \psi$ does not vanish and moreover,  it factors through $S_b$.
Therefore,  $\xi \psi\in \Re_A^{r_a}(P_{b},I_{b})$. Finally, by Lemma \ref{clau2.5}
we get that $r_a\leq r_b$.
\end{proof}

A similar result holds if $I_a$ is a simple injective $A$-module.
Below, we state the result.

\begin{prop}\label{nofuente}
Let $A\simeq kQ_A/I_A$ be a representation-finite algebra.
Assume that $Q_A$ has at leat one source  and that
$I_a$ is a simple injective $A$-module.  Then, there
is a vertex $b\neq a$ such that $r_a\leq r_b.$
\end{prop}

\begin{obs}
\emph{Let $A\simeq kQ_A/I$ be an algebra and $T[a]$ be the APR-tilting module
corresponding to a  sink $a$ of $Q_A$. By the construction of the module $T[a]$
and Proposition \ref{nopozo}, it follows that there is a vertex $u\in(R_A)_0$
such that $P_u\in \mbox{add}\,T[a]$.}
\end{obs}

The next lemma is a particular case of Theorem \ref{Tsepar}.

\begin{lema}\label{APRgral}
Let $A\simeq kQ_A/I_A$ be an algebra and $a\in(Q_A)_0$ be a sink. Let $T[a]$ be the APR-tilting $A$-module corresponding to the
vertex $a$ and $B=\emph{End}_A(T[a])$. If $B$ is representation-finite
then $r_A\leq r_B$.
\end{lema}

In case that $T[a]$ is an APR-tilting $A$-module, with $a$ a free sink of $Q_A$ we prove that the nilpotency indices
of $\Re(\mbox{mod}\,A)$ and $\Re(\mbox{mod}\,B)$ coincide. More precisely, Lemma \ref{APRgral} and Lemma \ref{APRlibre} give us Theorem B.

\begin{lema}\label{APRlibre}
Let $A\simeq kQ_A/I_A$ be a representation-finite algebra and $a\in(Q_A)_0$
a free sink.
Let $T[a]$ be the APR-tilting $A$-module corresponding to the vertex $a$ and $B=\emph{End}_A(T[a])$.
Then $r_A=r_B$.
\end{lema}

\begin{proof}
Let $A$ be a representation-finite algebra. Since $B=\mbox{End}_A(T[a])$,
where $T[a]$ is the APR-tilting $A$-module corresponding to the free sink $a$,
then $B$ is also representation-finite.

By Lemma \ref{APRgral}, we know that $r_A\leq  r_B.$
We claim that $r_A= r_B.$ In fact, if $r_A\leq r_B-1$ then
since $a$ is a free sink of $(Q_A)_0$, by \cite[Lemma p. 178]{As} we know that
$a'$, the vertex corresponding to $a$ in $Q_B$, is a source. Then,
by Proposition \ref{nofuente}, there is a vertex $b'\in (Q_B)_0$, with $b'\neq a'$,
such that $r_{b'}=r_B-1$. Consider a non-zero morphism $\til{\phii}_{b'}:{P'}_{b'}\fle {I'}_{b'}$
that factors through ${S'}_{b'}$ in $\mbmB$.
Then, by Lemma \ref{clau2.5} we have that
$dp(\til{\phii}_{b'})={r_B-1}.$

Moreover, the indecomposable $B$-modules ${P'}_{b'}$  and ${I'}_{b'}$
belong to the subcategory $\Y(T[a])$. Then, by Brenner and Butler's theorem, there are
indecomposable $A$-modules $M,N\in \T(T[a])$ and a non-zero morphism $f:M\fle N$
such that ${P'}_{b'}=F(M)$, ${I'}_{b'}=F(N)$ and  $\til{\phii}_{b'}=F(f)$.
By Corollary \ref{BBcompos}, we know that
$dp(f)={r_B-1}$,
a contradiction to the fact that $\Re^{r_A}(\mbmA)=0$ and $r_A\leq r_B-1$.
Therefore, $r_A=r_B$.
\end{proof}

As an immediate consequence of Proposition  \ref{nopozo} and Proposition
\ref{nofuente}, we know that to compute the nilpotency index of the radical of a module
category of a representation-finite algebra, it is enough to analyze
the vertices of $Q_A$ which are neither sinks nor sources. Below we state such a result.

\begin{teo}\label{nilpo}
Let $A=kQ_A/I_A$ be a representation-finite algebra.
Consider $\mathcal{V}$ the subset of vertices of $Q_A$ which are neither sinks nor sources.
Assume that $\mathcal{V}\neq \emptyset$ and that for every $a\in \mathcal{V}$, we write
$r_a=d_{r}(\iota _{a})+d_{l}(\theta _{a})$, where  $\iota_a:\emph{rad}\, P_a\fle P_a$ and $\theta_a:I_a\fle I_a/\emph{soc}\,I_a$
are irreducible morphisms, with $P_a$ and $I_a$ the projective and the injective modules corresponding to the vertex $a$.
Then, $r_A=\emph{max}\{r_a,\text{ con } a\in \mathcal{V}\}+1$ is the nilpotency index of $\Re(\mbmA)$.
\end{teo}

\begin{obs}
\emph{If the set $\mathcal{V}$ is empty, that is, all the vertices of $Q_A$ are sinks or sources, then the algebra is hereditary.
In the next section, we study the nilpotency index of the radical of
the module category of a representation-finite hereditary algebra, see Theorem \ref{ppalher}.}
\end{obs}

\section{Application to iterated tilted algebras}

The aim of this section is to study a bound for the nilpotency index
of the radical of the module category of an iterated tilted algebra of Dynkin type.

We start proving some essential results for our further considerations.

\begin{prop}\label{gamaHlong}
Let $H$ be a representation-finite hereditary algebra.
Then, $\Gamma_H$ is a component with length.
\end{prop}

\begin{proof}
In \cite[Section 4.3]{BG}, it was proved that the orbit graph $O(\Gamma_H)$
is a deformation retract of $\Gamma_H$. Therefore, $O(\Gamma_H)$
is a tree and consequently $\Gamma_H$ is a simply connected translation quiver.
Hence, $\Gamma_H$ is a component with length.
\end{proof}

\begin{lema}\label{camiglar}
Let $H \simeq kQ_H$ be a representation-finite hereditary algebra. Then, $r_a=r_b$,
for all vertices $a,b\in (Q_H)_0$.
\end{lema}

\begin{proof}
Let $a\in (Q_H)_0$. We know that $r_a$ is the length
of a path of irreducible morphisms between indecomposable modules
from the projective $P_a$ to the injective $I_a$,
going through the simple $S_a$.
Since $\Gamma_{\mbox{mod}\,H}$ is with length, it is
enough to analyze any path from $P_a$ to $I_a$.

Let $i,j\in (Q_{H})_0$ be vertices such that there is an arrow from $i$ to $j$.
Then, there are irreducible
morphisms $P_{j}\flechad P_{i}$ and $I_{j}\flechad I_{i}$,
where $P_j$ and $P_i$ are the projective $H$-modules corresponding to the vertices
$j$ and $i$ respectively; and $I_j$ and  $I_i$ are the
injective $H$-modules corresponding to the vertices $j$ and $i$, respectively.
Moreover, by \cite[Lemma 2.3]{C},
there is a non-zero path $P_{a}\rightsquigarrow I_{a}$, for all $a\in (Q_{H})_{0}$. Then,
in $\Gamma_{\mbox{mod}\,H}$ we have the following two paths of irreducible morphisms:

\begin{displaymath}
\xymatrix {& & I_j \ar[rd]  & \\
P_j  \ar[rd]\ar@{~>}[rru] & & & I_i \\
& P_i  \ar@{~>}[rru] & & }
\end{displaymath}

Assume that the path
$P_{j}\rightsquigarrow I_{j}$ has length $m$. Then,
the path $P_{j}\rightsquigarrow I_{j}\flechad I_{i}$ has
length $m+1$, because the arrow denote an irreducible morphism
between indecomposable modules. Since $\Gamma_{\mbox{mod}\,H}$ is
with length, the path  $P_{j}\flechad P_{i}\rightsquigarrow I_{i}$
has also length $m+1$ and hence, the path $P_{i}\rightsquigarrow I_{i}$
has length $m$.
Since $H$ is connected, then all
paths of the form $P_{a}\rightsquigarrow I_{a}$, with $a\in (Q_H)_0$
have length $m$, and therefore
$r_a=r_b $
for all $a, b\in (Q_H)_0$.
\end{proof}

By \cite[Theorem 4.11]{Z}, we recall the following result for a representation-finite hereditary algebra.

\begin{teo}\label{ppalher}
Let $H=kQ$ be a representation-finite hereditary algebra and let $r_H$
be the nilpotency index of $\Re(\emph{mod}\,H).$
\begin{enumerate}
\item[(a)] If $\overline{Q}=A_n$, then $r_H=n,$ for $n\geq 1$.
\item[(b)] If $\overline{Q}=D_n$, then $r_H=2n-3,$ for $n\geq 4$.
\item[(c)] If $\overline{Q}=E_6$, then $r_H=11.$
\item[(d)] If $\overline{Q}=E_7$, then $r_H=17.$
\item[(e)] If $\overline{Q}=E_8$, then $r_H=29.$
\end{enumerate}
\end{teo}

Now, we prove a necessity condition
of Proposition \ref{propinc}, when $B$ is a hereditary algebra.

\begin{prop}
Let $A\simeq kQ_A/I_A$ be an algebra
and $B=\emeT$ be a representation-finite hereditary algebra,
with $T$ a tilting $A$-module.
Assume that there exists a vertex $u\in ({R_A})_0$
such that $P_u\in \emph{add}\,T$.
Then, $r_A=r_B$ if and only if there is a nonzero path
of irreducible morphisms
$F(P_u)\fle {X'}_1\fle \dots \fle {X'}_{r_A-2}\fle F(I_u)$
of length $r_A-1.$
\end{prop}

\begin{proof}
Since $B$ is a representation-finite hereditary algebra, then $T$
is a separating tilting $A$-module and $A$ is also a representation-finite algebra.
Consider $r_A$ and $r_B$ the nilpotency indices of $\Re(\mbmA)$ and $\Re(\mbmB)$, respectively.
Assume that there exists a vertex $u\in (R_A)_0$ such that $P_u\in \mbox{add}\,T$.
The sufficient condition follows from Proposition \ref{propinc}.

Conversely, by hypothesis there exists in $\mbmB$ a nonzero path of irreducible morphisms
$F(P_u)\fle {X'}_1\fle \dots \fle {X'}_{r_A-2}\fle F(I_u)$
of length $r_A-1$. We shall prove that $r_A=r_B$. We know
that $F(P_u)={P'}_{u^*}$ and $F(I_u)={I'}_{u^*}$ are the
projective and the injective $B$-modules corresponding to the vertex
${u^*}\in (Q_B)_0$.

On the other hand, since $B$ is an hereditary algebra, by Lemma \ref{camiglar}
we have that any path of irreducible morphism from a projective to the injective
corresponding to the same vertex has length $r_B-1$. In particular,
the above path has length $r_B-1$. Therefore, $r_A=r_B$.
\end{proof}

The following result was proved by Happel in \cite{Ha}.

\begin{teo}\emph{\cite[Teorema 6.2]{Ha}}\label{Ha6.2}
Let ${\Delta}$ be a Dynkin quiver and $A$ be an iterated tilted algebra of type $\Delta.$
Then, there exist a sequence of algebras
$A=A_0,A_1,\dots , A_m=k\Delta$  and a sequence of APR-tilting $A_i$-modules $T^{(i)}$,
for $0\leq i<m$, such that  $A_{i+1}=\emph{End}_{A_i}T^{(i)}$.
\end{teo}

As a consequence of Lemma \ref{APRgral} and Theorem \ref{ppalher} we get an upper bound for the nilpotecy index of
an iterated tilted algebra.

\begin{coro}\label{inclite}
Let $\Delta$ be a Dynkin quiver and $A$ be an iterated tilted algebra of type $\Delta.$
Consider $r_A$ the nilpotency index of $\Re_A(\emmA)$.
\begin{enumerate}
\item[(a)] If $\overline{\Delta}=A_n$, then $r_A\leq n$ for $n\geq 1$.
\item[(b)] If $\overline{\Delta}=D_n$, then  $r_A\leq 2n-3$ for $n\geq 4$.
\item[(c)] If $\overline{\Delta}=E_6$, then  $r_A\leq 11.$
\item[(d)] If $\overline{\Delta}=E_7$, then  $r_A\leq 17.$
\item[(e)] If $\overline{\Delta}=E_8$, then  $r_A\leq 29.$
\end{enumerate}
\end{coro}

The next example shows that there exists iterated tilted algebras such that
the nilpotency index of the radical of its module category
have the maximum value of the bound  given in Corollary \ref{inclite}.
Moreover, two iterated tilted algebras of the same Dynkin type, may have different
nilpotency indices.

\begin{ej} \label{ejemplo1}
\emph{Let $A_1$ be the algebra given by the bound quiver }
{\small\begin{displaymath}
    \xymatrix  @!0 @R=0.7cm  @C=0.7cm {
       1\ar[rr]^{\alpha}&\ar@{.}@/^/[rr] &2\ar[rr]^{\beta} &&3\ar[rr]^{\gamma}&&4\ar[rr]^{\delta}&& 5 & }
\end{displaymath}}

\noindent \emph{\noindent with $I_1=<\beta\alpha>$ and  $A_2$ be the algebra given by the bound quiver }
{\small\begin{displaymath}
    \xymatrix  @!0 @R=0.7cm  @C=0.7cm {
       1\ar[rr]^{\alpha}&\ar@{.}@/^/[rr] &2\ar[rr]^{\beta} &&3\ar[rr]^{\gamma}&\ar@{.}@/^/[rr]&4\ar[rr]^{\delta}&& 5}
\end{displaymath}}

\noindent \emph{\noindent with  $I_2=<\beta\alpha,\delta\gamma>$.}

\noindent \emph{Observe that both algebras are iterated tilted algebras of type $A_5$. By
Corollary \ref{inclite}, we know that $r_{A_i}\leq 5$ for $i=1, 2$.
Computing each nilpotency index, we get that $r_{A_1}=5$ but  $r_{A_2}=4$.}
\end{ej}

An algebra $A$ is called
a \emph{directed algebra} if each indecomposable A-module is directed, that is,
it does not belong to any cycle in $\mbmA$.

\begin{prop}\emph{\cite[Corollary 3.6]{Ha}}\label{Hap}
Let $A$ be a representation-finite iterated tilted algebra. Then,
$A$ is a directed algebra.
% $\emmA$ is directed.
\end{prop}

Recall that an algebra $A$ is standard if the
category $\mbiA$ is equivalent to the mesh category of $\Gamma_A$, see \cite[2.5]{ChaTre}. 
It is known by \cite{BrG} that a directed algebra is standard.

Following \cite{ChaTre}, by a \emph{bypass} of an irreducible morphism $f:X\fle Y$,
with $X$ and $Y$ indecomposable modules,  we mean a path
$X \stackrel{t_1}\longrightarrow Y_1\stackrel{t_2}\longrightarrow
 Y_2 \longrightarrow \cdots \longrightarrow Y_n
\stackrel{t_{n+1}}\longrightarrow Y$ in $\mbiA$ of length $n+1 \geq 2$, where $t_1$ and $t_{n+1}$
 are irreducible morphisms,
$X \ncong Y_n$ and $Y\ncong Y_1$.

In \cite[Corollary 2.9]{ChaTre}, the authors proved that if we consider a standard algebra $A$ and a nonzero composition of $n$ irreducible
morphisms between indecomposable modules in $\Re_A^{n+1}$, then there exists a bypass or a cycle in $\Gamma_A$.  Since,
by \cite{S}, there are no bypasses in a
directed algebra, then we have the following result.  

\begin{lema}\emph{\cite[Corollary 2.11]{ChaTre}}\label{ctdir}
Let A be a directed algebra. Let $h_i:X_i\fle X_{i+1}$ be irreducible morphisms between indecomposable modules,
for $1\leq i\leq n$. Then $h_n \dots h_1\in \Re_A^{n+1}(X_1,X_{n+1})$
if and only if $h_n \dots h_1=0.$
\end{lema}

Our next result is an immediately consequence of Proposition \ref{Hap} and   Lemma \ref{ctdir}.

\begin{prop}
Let $\Delta$ be a Dynkin quiver and $A$ be an iterated tilted algebra of type ${\Delta}$.
Let $h_i:X_i\fle X_{i+1}$ be irreducible morphisms between indecomposable modules,
for $1\leq i\leq n$. Then $h_n \dots h_1\in \Re_A^{n+1}(X_1,X_{n+1})$
if and only if $h_n \dots h_1=0.$
\end{prop}

\end{document}